\documentclass[10pt,a4paper]{amsproc}

\theoremstyle{plain}
\newtheorem{theorem}{Theorem}
\newtheorem{lemma}[theorem]{Lemma}
\theoremstyle{definition}
\newtheorem{definition}[theorem]{Definition}

\author{V.~E.~Nazaikinskii}
\address{A.~Ishlinsky Institute for Problems in Mechanics,
Moscow\\ Moscow Institute of Physics and Technology, Dolgoprudny,
Moscow District}
\email{nazaikinskii@yandex.ru}

\date{10.05.2013}

\title[On a $KK$-Theoretic Counterpart of Relative Index Theorems]%
{On a $KK$-Theoretic Counterpart\\ of Relative Index Theorems}

\subjclass[2010]{46L80 (Primary); 19K35, 58J20 (Secondary)}

\begin{document}

\maketitle

\begin{abstract}
Relative index theorems, which deal with what happens with the
index of elliptic operators when cutting and pasting, are abundant
in the literature. It is desirable to obtain similar theorems for
other stable homotopy invariants, not the index alone. In the
spirit of noncommutative geometry, we prove a full-fledged
``relative index'' type theorem that compares certain elements of
the Kasparov $KK$-group $KK(A,B)$.
\end{abstract}

\subsection*{Introduction}

Relative index theorems, which deal with what happens with the
index of elliptic operators when cutting and pasting, are abundant
in the literature. Here we only mention the famous Gromov--Lawson
relative index theorem~\cite{GrLa1} for Dirac operators on complete
noncompact Riemannian manifolds and refer the reader for further
examples and bibliography to the papers~\cite{R:NaSt7,AAA}, where a
locality theorem for the relative index was proved in a rather
general setting, which not only covered many earlier-known special
cases but also permitted one to obtain a number of index formulas
for elliptic differential operators and even for Fourier integral
operators on manifolds with singularities~\cite{NSScS99}. It is
however desirable to obtain similar theorems for other stable
homotopy invariants, not the index alone. One of the first steps in
this direction was made much earlier by Bunke~\cite{Bun95}, who
considered Dirac operators on a complete noncompact Riemannian
manifold in section spaces of bundles of projective Hilbert
$B$-modules (e.g., see \cite{MisFom79}), where $B$ is a
$C^*$\nobreakdash-algebra, and obtained a relative index theorem
for such operators, the index being an element of the $K$-group
of~$B$. In the present paper, in the spirit of noncommutative
geometry, we prove a full-fledged ``relative index'' type theorem
that compares certain elements of the Kasparov $KK$-group
$KK(A,B)$. In contrast to~\cite{Bun95}, where $KK$-groups are used
with $A$ being an algebra of functions on the manifold and the
answers are only stated in terms of elements of
$K_*(B)=KK(\mathbb{C},B)$, we admit an arbitrary
\textit{noncommutative} unital $C^*$-algebra $A$ and do not
restrict ourselves to the index, even to the $K_*(B)$-valued one.

We freely use notions and notation related to $C^*$-algebras and
$KK$-theory (e.g., see \cite{Ped1,HiRo1,Bla1} and the literature
cited therein). Full proofs will be given elsewhere.

\subsection{Algebra $A$ and a partition of unity}\label{1}
Let $A$ be a unital $C^*$-algebra, and let $J_1,J_2\subset A$ be
two (closed, two-sided, $*$-)ideals such that $J_1+J_2=A$. By $J$
we denote the intersection of these ideals, $J=J_1\cap J_2$.
\begin{lemma}\label{le1}
There exist self-adjoint positive elements $\psi_1\in J_1$ and
$\psi_2\in J_2$ with
\begin{equation}\label{eq1}
    \psi_1^2+\psi_2^2=1,\qquad [\psi_1,\psi_2]=0.
\end{equation}
\end{lemma}
\begin{proof}
Since $J_1+J_2=A$, it follows that there exists an element $\chi\in
J_1$ such that $1-\chi\in J_2$; taking the real part, we can assume
that $\chi=\chi^*$. By functional calculus, the element
$\chi^2+(1-\chi)^2$ is positive and invertible, and we set
$\psi_1=\lvert\chi\rvert(\chi^2+(1-\chi)^2)^{-1/2}$ and
$\psi_2=\lvert 1-\chi\rvert(\chi^2+(1-\chi)^2)^{-1/2}$.
\end{proof}

\subsection{Kasparov modules}\label{2}

Let $B$ be another $C^*$-algebra, and let $x=(H,\rho,F)$ be a
\textit{Kasparov module} for $(A,B)$, i.e., a triple consisting of
a Hilbert module~$H$ over~$B$, a homomorphism $\rho\colon
A\to\mathbb{B}(H)$ of~$A$ into the $C^*$-algebra $\mathbb{B}(H)$ of
adjointable operators on~$H$, and an operator $F\in\mathbb{B}(H)$
such that
\begin{equation}\label{eq2}
    [F,\rho(a)]\sim 0,\qquad \rho(a)(F^2-1)\sim0,\qquad
    \rho(a)(F-F^*)\sim0
\end{equation}
for every $a\in A$, where we write $C\sim D$ if $C-D$ lies in the
ideal $\mathbb{K}(H)\subset \mathbb{B}(H)$ of ``compact''
operators. We only consider Kasparov modules in which the
homomorphism $\rho$ is unital (and refer to these as
\textit{unital} Kasparov modules); in this case, the factor
$\rho(a)$ can be dropped in the second and third conditions in
\eqref{eq2}. The $(A,B)$-sub-bimodules
\begin{equation}\label{eq3}
    H_1=\rho[J_1]H,\quad H_2=\rho[J_2]H,
    \quad H_0=\rho[J]H
\end{equation}
of $H$ are automatically closed (the proof is similar to that
in~\cite[pp.~25--26]{HiRo1}) and hence are simultaneously Hilbert
$B$-modules. It is easily seen that $H_0=H_1\cap H_2$. Indeed, the
inclusion $H_0\subset H_1\cap H_2$ is obvious. Next, let $\xi\in
H_1\cap H_2$. Then $\rho(u_\lambda)\xi\to\xi$ and
$\rho(v_\mu)\xi\to\xi$, where $\{u_\lambda\}$ and $\{v_\mu\}$ are
approximate units for $J_1$ and $J_2$, respectively. Now it follows
from the inequality $\lVert \rho(u_\lambda
v_\mu)\xi-\xi\rVert\le\lVert \rho(v_\mu)\xi-\xi\rVert+\lVert
\rho(u_\lambda)\xi-\xi\rVert$ that there exists a subsequence of
$\{\rho(u_\lambda v_\mu)\xi\}$ that converges to $\xi$, and hence
$\xi\in H_0$, because $u_\lambda v_\mu\in J$ and $\rho(u_\lambda
v_\mu)\xi\in H_0$.

The $(A,B)$-bimodule $H$ is naturally isomorphic to the quotient
$(H_1\oplus H_2)\slash\Delta$, where $\Delta= \{(\xi_1,\xi_2)\in
H_1\oplus H_2\colon \xi_1=-\xi_2\in H_0\}$. The isomorphism is
induced by the mapping $\alpha\colon H_1\oplus H_2\to H$,
$(\xi_1,\xi_2)\mapsto\xi_1+\xi_2$, with the inverse being induced
by $\beta\colon H\to H_1\oplus H_2$,
$\xi\mapsto(\rho(\psi_1^2)\xi,\rho(\psi_2^2)\xi)$.

\subsection{Cutting and pasting}\label{3}

Let $\widetilde x=(\widetilde H,\widetilde \rho,\widetilde F)$ be
another unital Kasparov $(A,B)$-module, and let $\widetilde H_j$,
$j=0,1,2$, be the Hilbert submodules of $\widetilde H$ defined as
in \eqref{eq3}.
\begin{definition}\label{de3}
We say that $x$ and $\widetilde x$ \textit{agree on $J$} if there
is a unitary (in the sense of Hilbert modules over $B$) isomorphism
$T\colon H_0\to\widetilde H_0$ of $(A,B)$-bimodules such that, for
arbitrary $c,d\in J$, one has
\begin{equation}\label{eq4}
    T \rho(c) F \rho(d) \sim
    \widetilde\rho(c) \widetilde F \widetilde\rho(d) T.
\end{equation}
(Note that $\rho(c) F \rho(d)\in \mathbb{B}(H_0)$ and
$\widetilde\rho(c) \widetilde F \widetilde\rho(d)\in
\mathbb{B}(\widetilde H_0)$ are well defined.)
\end{definition}
Assume that $x$ and $\widetilde x$ agree on $J$. Our aim is to use
some sort of \textit{cutting-and-pasting procedure} to define a
unital Kasparov $(A,B)$-module $x\diamond\widetilde x$ that agrees
with $x$ on $J_1$ and with $\widetilde x$ on $J_2$. To this end,
consider the $(A,B)$-bimodule
\begin{equation}\label{eq5}
 H\diamond \widetilde H=(H_1\oplus\widetilde H_2)\Big\slash
 \{(\xi_1,\xi_2)\colon \xi_1\in H_0,\;
 \xi_2\in\widetilde H_0,\; T\xi_1+\xi_2=0\}.
\end{equation}
The elements of $H\diamond \widetilde H$ will be denoted by
$\xi=[(\xi_1,\xi_2)]$, and the action of $A$ on $H\diamond
\widetilde H$ will be denoted by $\rho\diamond\widetilde\rho$,
$(\rho\diamond\widetilde\rho)(\varphi)\xi= [(\rho(\varphi)\xi_1,
\rho(\varphi)\xi_2)]$. Note that $H_1$ and $\widetilde H_2$ are
naturally embedded in $H\diamond \widetilde H$ (the embeddings are
induced by those of $H_1$ and $\widetilde H_2$ in the direct sum
$H_1\oplus\widetilde H_2$), and if we identify $H_1$ and
$\widetilde H_2$ with their images under these embeddings,
\begin{equation}\label{eq6}
    H_1\simeq(H\diamond\widetilde H)_1\equiv
    (\rho\diamond\widetilde\rho)[J_1](H\diamond\widetilde H) ,
    \quad
    \widetilde H_2\simeq(H\diamond\widetilde H)_2\equiv
    (\rho\diamond\widetilde\rho)[J_2](H\diamond\widetilde H),
\end{equation}
then, for arbitrary $\xi\in H\diamond \widetilde H$ and
$\varphi_j\in J_j$, $j=1,2$, we have
$(\rho\diamond\widetilde\rho)(\varphi_1)\xi\simeq\rho(\varphi_1)\xi_1
+T^*\widetilde\rho(\varphi_1)\xi_2\in H_1$ and
$(\rho\diamond\widetilde\rho)(\varphi_2)\xi\simeq
T\rho(\varphi_2)\xi_1 +\widetilde\rho(\varphi_2)\xi_2\in\widetilde
H_2$.

From now on, to simplify the notation, we identify $H_0$ and
$\widetilde H_0$ via $T$, accordingly suppress $T$ in all the
formulas, and also write simply $\varphi$ instead of
$\rho(\varphi)$, $\widetilde\rho(\varphi)$, or
$(\rho\diamond\widetilde\rho)(\varphi)$. This will not lead to a
misunderstanding even if several representations are involved,
because which is meant is always clear from the context.

\begin{lemma}\label{le4}
The formula
\begin{equation*}
    \langle \xi,\eta\rangle_{H\diamond \widetilde H}
     =\bigl\langle \psi_1^2 \xi,
     \psi_1^2 \eta\bigr\rangle_H
     +\bigl\langle
     \psi_2^2 \xi ,
      \psi_2^2 \eta \bigr\rangle_{\widetilde H}
     +2\bigl\langle \psi_1\psi_2 \xi ,
         \psi_1\psi_2\eta \bigr\rangle_H,
\end{equation*}
where $\langle\,\boldsymbol\cdot\,,\,\boldsymbol\cdot\,\rangle_H$
and
$\langle\,\boldsymbol\cdot\,,\,\boldsymbol\cdot\,\rangle_{\widetilde
H}$ are the $B$-valued inner products on $H$ and $\widetilde H$,
respectively, specifies a well-defined $B$-valued inner product on
$H\diamond \widetilde H$, which makes $H\diamond \widetilde H$ a
Hilbert $B$-module and the action of $A$ on $H\diamond \widetilde
H$ a unital $*$-homomorphism $\rho\diamond \widetilde\rho\colon
A\to \mathbb{B}(H\diamond\widetilde H)$.
\end{lemma}
Now set
\begin{equation}\label{eq8}
    F\diamond \widetilde F=\psi_1 F\psi_1+\psi_2 \widetilde F\psi_2\colon
    H\diamond\widetilde H\longrightarrow H\diamond\widetilde H.
\end{equation}
This is well defined. Indeed, for example, if $\xi \in
H\diamond\widetilde H$, then, in view of our identifications,
$\psi_1\xi\in (H\diamond\widetilde H)_1=H_1\subset H$, hence
$F\psi_1\xi$ is a well-defined element of $H$, and hence
$\psi_1F\psi_1\xi\in H_1\subset H\diamond\widetilde H$ is a
well-defined element of $H\diamond\widetilde H$.

\begin{theorem}\label{t1}
The triple $x\diamond\widetilde x=(H\diamond\widetilde
H,\rho\diamond \widetilde\rho,F\diamond \widetilde F)$ is a unital
Kasparov $(A,B)$-module, which is independent modulo
\textup{``}compact\textup{''} perturbations of the choice of the
partition of unity \eqref{eq1} and agrees with $x$ on $J_1$ and
with $\widetilde x$ on $J_2$.
\end{theorem}
In a similar way, one defines the Kasparov module $\widetilde
x\diamond x$, which agrees with $\widetilde x$ on $J_1$ and with
$x$ on $J_2$.

\subsection{Main result}\label{4}

Now we are in a position to state the main result of the paper. For
a Kasparov $(A,B)$-module $y$, let $[y]\in KK(A,B)$ be the
corresponding class in Kasparov's $KK$-theory.
\begin{theorem}
Under the above assumptions, one has
\begin{equation}\label{eq9}
    [x]+[\widetilde x]=[x\diamond \widetilde x]+[\widetilde x \diamond x].
\end{equation}
\end{theorem}
Let us give a sketch of the proof. First, we note that, under our
identifications, the formula
\begin{equation*}
    \begin{pmatrix}
      \xi \\
      \eta
    \end{pmatrix}\longmapsto
    \begin{pmatrix}
      \psi_1\xi+\psi_2\eta \\
      -\psi_2\xi+\psi_1\eta
    \end{pmatrix}
\end{equation*}
gives a well-defined unitary isomorphism
\begin{equation*}
    U\colon H\oplus\widetilde H\longrightarrow (H\diamond\widetilde H)\oplus (\widetilde H\diamond H)
\end{equation*}
of Hilbert $B$-modules such that
\begin{equation}\label{eq10}
    U(F\oplus\widetilde F)U^*\sim (F\diamond\widetilde F)\oplus (\widetilde F\diamond F).
\end{equation}
Moreover, this mapping is an isomorphism of $(A,B)$-bimodules,
where the action of $A$ on $(H\diamond\widetilde H)\oplus
(\widetilde H\diamond H)$ is the direct sum of the actions
$\rho\diamond\widetilde\rho$ and $\widetilde\rho\diamond\rho$,
while the action of $A$ on $H\oplus\widetilde H$ is given by the
matrix
\begin{equation}\label{eqA}
    \widehat\rho(a)=\begin{pmatrix}
      \psi_1a\psi_1+\psi_2a\psi_2 & \psi_2a\psi_1-\psi_1a\psi_2 \\
      \psi_1a\psi_2-\psi_2a\psi_1 & \psi_1a\psi_1+\psi_2a\psi_2
    \end{pmatrix},\qquad a\in A.
\end{equation}
Note that the right-hand side is well defined because the
off-diagonal entries lie in $J$ and hence take $H$ as well as
$\widetilde H$ to $H_0$, which lies in both. Moreover,
$[\widehat\rho(a),F\oplus\widetilde F]\sim0$, because $F$ and
$\widetilde F$ agree on $J$.  Thus, it remains to prove that the
Kasparov modules $(H\oplus\widetilde
H,\rho\oplus\widetilde\rho,F\oplus\widetilde F)$ and
$(H\oplus\widetilde H,\widehat\rho,F\oplus\widetilde F)$ define the
same element in $KK(A,B)$. To this end, we construct a homotopy
$(H\oplus\widetilde H,\widehat\rho_t, F\oplus\widetilde F)$ of
Kasparov modules such that
$\widehat\rho_0=\rho\oplus\widetilde\rho$ and
$\widehat\rho_1=\widehat\rho$. Namely, for $a\in A$ we set
\begin{equation}\label{eqB}
        \widehat\rho_t(a)=\begin{pmatrix}
      \psi_{1t}a\psi_{1t}+\psi_{2t}a\psi_{2t} & \psi_{2t}a\psi_{1t}-\psi_{1t}a\psi_{2t} \\
      \psi_{1t}a\psi_{2t}-\psi_{2t}a\psi_{1t} & \psi_{1t}a\psi_{1t}+\psi_{2t}a\psi_{2t}
    \end{pmatrix},
\end{equation}
where $\psi_{1t}=t\psi_1$ and $\psi_{2t}=\sqrt{1-t^2\psi_1^2}$.

First, we should prove that the operator \eqref{eqB} is well
defined on $H\oplus\widetilde H$. To this end, it suffices to show
that the off-diagonal entries lie in $J$. We have
$\psi_{2t}=\sqrt{1-t^2}+\bigl(\sqrt{1-t^2+t^2\psi_2^2}-\sqrt{1-t^2}\bigr)$;
the term in parentheses lies in $J_2$ by functional calculus, and
we obtain $\psi_{2t}a\psi_{1t}-\psi_{1t}a\psi_{2t}\equiv
t\sqrt{1-t^2}[a,\psi_1]\mod J$. It remains to note that
$[a,\psi_1]\in J_1$ and $[a,\psi_1]=[a,\sqrt{1-\psi_2^2}]
=[a,\sqrt{1-\psi_2^2}-1]\in J_2$, again by functional calculus, so
that $[a,\psi_1]\in J$ and we are done.

Now the properties $\psi_{1t}^2+\psi_{2t}^2=1$ and
$[\psi_{1t},\psi_{2t}]=0$ imply that $\widehat\rho_t$ is indeed a
homomorphism and that $[\widehat\rho_t(a),F\oplus\widetilde
F]\sim0$. (Here we again use the fact that $F$ and $\widetilde F$
agree on $J$.) This completes the proof.

\end{document}